\documentclass[a4paper,12pt]{amsart}
\usepackage{amsmath, amsfonts, amssymb}
\usepackage[english]{babel}
\newtheorem{thm}{Theorem}[section]

\newtheorem{lm}[thm]{Lemma}

\numberwithin{equation}{section}

\begin{document}

\title{Description of facially symmetric spaces with unitary tripotents}

\author[Kudaybergenov]{Karimbergen Kudaybergenov}
\address{Ch. Abdirov 1, Department of Mathematics, Karakalpak State University, Nukus 230113, Uzbekistan}
\email{karim2006@mail.ru}

\author[Jumabek Seypullaev]{Jumabek Seypullaev} 
\address{V.I.Romanovskiy Institute of
Mathematics, Uzbekistan Academy of Sciences, 81, Mirzo Ulughbek
street, 100170  Tashkent,   Uzbekistan}
\email{jumabek81@mail.ru}
%
%

\date{}
\maketitle

\begin{abstract} 
We give a description of finite-dimensional real
neutral strongly facially symmetric spaces with the property JP. We also
prove that if \(Z\) is a real neutral strongly facially symmetric
with an unitary tripotents, then the space \(Z\) is isometrically isomorphic
to the space \(L_1(\Omega,\Sigma, \mu),\) where \((\Omega,\Sigma,
\mu)\) is a measure space having the direct sum property.

\end{abstract}

\section{Introduction}\label{sec:intro}
A geometric characterization of state spaces of operator algebras is  important problem of the theory of operator algebras
In the middle of 1980s, Friedman and Russo wrote the paper \cite{fr1, fr2} related to
this problem, in which they introduced facially symmetric spaces, largely for the purpose of obtaining
a geometric characterization of the predual spaces of JBW*-triples admitting an algebraic structure.
Many of the properties required in these characterizations are natural assumptions for state spaces of
physical systems. Such spaces are regarded as a geometric model for states of quantum mechanics.
In \cite{fr3} was proved that the preduals of complex von Neumann algebras and, more generally, complex JBW*-triples are
neutral strongly facially symmetric spaces.

The project of classifying facially symmetric spaces was initiated in \cite{fr4}, where a geometric characterization
of complex Hilbert spaces and complex spin factors was given. Moreover, there were described the JBW*-triples of ranks 1 and 2
and Cartan factors of types 1 and 4. Afterwards, Friedman and Russo obtained a
description of atomic facially symmetric spaces \cite{fr5}. Namely, they showed that a neutral strongly facially
symmetric space is linearly isometric to the predual of one of the Cartan factors of types 1-6 provided
that it satisfies four natural physical axioms, which hold for the predual spaces of JBW*-triples.

In paper \cite{fr6}, Neal and Russo found geometric conditions
under which a facially symmetric space is isometric to the predual
of a JBW*-triple. In particular, they proved that any neutral
strongly facially symmetric space decomposes into a direct sum of
atomic and nonatomic strongly facially symmetric spaces. In
\cite{ikts}, a complete description of strongly facially symmetric
spaces that are isometrically isomorphic to a predual space
atom commutative von Neumann algebra.

In \cite[Theorem 3.2]{yik} was established that if $Z$ is a real neutral strongly
facially symmetric space with unitary tripotent such that any maximal tripotent is unitary, then $Z$ is isometrically isomorphic to $L_1$-space.
In the present paper shows that the condition "every maximal
tryptotent is unitary "in the above theorem is superfluous.
We also we give
a description of finite-dimensional real neutral SFS-spaces with the property JP.


\section{Facially symmetric spaces.}

In this section we shall recall some
basic facts and notation about facially symmetric spaces (see for details
\cite{fr1, fr2}).

Let $Z$ be a real or complex normed space.
Elements   $f, g \in Z$ are orthogonal, notation $f\diamond g,$ if $\|f+g\|=\|f-g\|=\|f\|+\|g\|.$
Subset $S,T\subset Z$ are said to be orthogonal, notation $S\diamond T,$ if $f\diamond g$ for all $(f, g)\in S\times T.$
 A norm exposed face of the unit ball $Z_1=\{f\in Z: \|f\|\leq 1\}$ of $Z$ is a non-empty set (necessarilly $\neq Z_1$) of
 the form $F_u=\{f\in Z_1: u(f)=1\},$  where $u\in Z^{*},$ $\|u\|=1.$
 Recall that a face $F$ of a convex set
 $K$ is a non-empty convex subset of $K$ such that if $f\in F$ and $g, h\in K$ satisfy $f=\lambda g+(1-\lambda)h$
 for some $\lambda\in (0, 1),$ then $g, h\in F.$ In particular, an extreme point of $K$ is face of $K.$
 An element $u\in Z^{*}$ is called a projective unit if  $\|u\|=1$  and $\langle u, F^{\diamond}_u\rangle=0.$
Here, for any subset $S,$ $S^{\diamond}$ denotes the set of all elements orthogonal to each elements of $S.$
A norm exposed face
$F_u$ in $Z_1$ is said to be a symmetric face if there is
a linear isometric symmetry $S_u$ of $Z$  onto $Z,$  with $S^{2}_u=I$
 such that fixed point set of $S_u$ is $(\overline{\mathrm{sp}}F_{u})\oplus F_{u}^\diamond.$

A real or complex normed space $Z$  is said weakly facially symmetric (WFS) if every norm exposed face in $Z_1$ is symmetric.

For each symmetric face $F_u$ we defined contractive projections $P_k(F_u), k=0, 1, 2$ on $Z$ as follows.
First, $P_1(F_u)=(I-S_u)/2$ is the projection on the $-1$ eigenspace of $S_u.$ Next we define
$P_2(F_u)$ and $P_0(F_u)$ as the projections of $Z$ onto  $\overline{\mathrm{sp}}F_u$ and  $F^\diamond_u,$ respectively.
A geometric tripotent is a projective unit $u\in Z^{\ast}$ with the property that $F_u$  is a symmetric face and
$S^{\ast}_uu=u$ for a symmetry
 $S_u$ corresponding to  $F_u.$
The projections $P_k(F_u)$ are called geometric Peirce projections.

 By $\mathcal{GU}$  and $\mathcal{SF}$  denote the collections of geometric tripotents and symmetric faces respectively, and the map
 $\mathcal{GU}\ni u\mapsto F_u\in \mathcal{SF}$ is a bijection \cite [Proposition 1.6] {fr2}. For each geometric tripotent $u$ in the dual
 of a WFS space $Z,$ we shall denote the geometric Peirce projections by $P_k(u)=P_k(F_u), k=0, 1, 2.$
 We set $U=Z^{\ast}, Z_k(u)=P_k(u)Z, U_k(u)=P_k(u)^{\ast}Z^{\ast}.$
 The Peirce decomposition
\begin{eqnarray*}
Z=Z_2(u)+Z_1(u)+Z_0(u), U=U_2(u)+U_1(u)+U_0(u).
\end{eqnarray*}
 Tripotents $u$ and $v$ are said to be orthogonal if $u\in U_0(v)$ (which implies $v\in U_0(u)$) or,
equivalently, $u\pm v\in \mathcal{GT}$ (see \cite [Lemma 2.5] {fr1}).
More generally, elements $a$ and $b$ of $U$ are said to be
orthogonal if one of them belongs to $U_2(u)$ and the other belongs to $U_0(u)$ for some geometric tripotent
$u.$

A contractive projection $Q$  on  $Z$ is said to be neutral if $\|Qf\|=\|f\|$ implies $Qf=f$  for each $f\in Z.$
A normed space $Z$ is said to be neutral if, for every symmetric face $F_u,$ the projection $P_2(u)$   is neutral.

 A WFS space $Z$  is strongly facially symmetric (SFS) if for every norm exposed face $F_u$
 in $Z_1$ and every $y\in Z^{*}$  with $\|y\|=1$ and  $F_u\subset F_y,$
 we have $S^{*}_uy=y,$  where $S_u$ denotes a symmetry associated with $F_u.$

 The principal examples of neutral complex strongly facially symmetric
spaces are preduals of complex JBW*-triples, in particular, the preduals of complex
von Neumann algebras (see~\cite{fr3}).
In a neutral strongly facially symmetric space $Z,$ every non-zero element
has a polar decomposition ~\cite[Theorem 4.3]{fr2}: for non-zero $f\in Z$ there exists
a unique geometric tripotent $v=v_f$ with $\langle v, f\rangle=\|f\|$ and $\langle v, f^{\diamond}\rangle = 0.$ If
$f, g\in Z,$ then $f\diamond g$ if and only if $v_f\diamond v_g,$ as follows from \cite [Corollary
1.3(b) and Lemma 2.1]{fr1}.
A partial ordering can be defined on the set of geometric tripotents as
follows: if $u, v \in \mathcal{GT},$ then $u\leq v$ if $F_u\subset F_v,$ or equivalently, by ~\cite[Lemma
4.2]{fr2}, $P_2(u)^{\ast}v = u,$ or $v-u$ is either zero or a geometric tripotent orthogonal
to $u.$ As is known ~\cite[Proposition 4.5]{fr2}, the set
$L_e=\{v\in \mathcal{GT}: v\leq e\}\cup \{0\}$ is a complete orthomodular lattice with orthocomplement
$u^{\perp}=e-u,$ with respect to the order $"\leq".$
The convex hull of the points $x_1, x_2, ..., x_{n+1}$  in general position is called an $n$-dimensional simplex.
In a finite-dimensional neutral SFS-space Z, every  norm exposed face is a simplex (see \cite[Proposition 1]{is3}).

\section{Description of finite-dimensional real strongly facially symmetric spaces}

One of the important concepts in SFS-spaces is the notion of rank
of space, which was  introduced in ~\cite{fr5}. In~\cite{is1} was
given a description  of a unit ball of an $n$-dimensional real
SFS-space of rank $n.$ In~\cite{s}  was obtained a
description of the unit ball of a reflexive SFS-space of rank 1.
In addition, in the  ~\cite{is3} was given a description of a
unit ball of an $n$-dimensional real neutral SFS-space of rank $n-1.$ The problem  of characterizing SFS-spaces with respect
to rank remains is  still open. In this section we give a
description of finite-dimensional real neutral SFS-spaces with the
property JP.

A facially symmetric space $Z$ is of rank $n (n = 1,2,...),$ notation $\mathrm{rank}  Z=n,$ if every
orthogonal family of geometric tripotents has cardinality at most $n,$ and if there is at least
one orthogonal family of geometric tripotents containing exactly $n$ elements  (see~\cite{fr5}).

A WFS space $Z$ satisfies JP (joint Peirce decomposition) if for any pair $u, v$ of orthogonal
geometric tripotents, we have $S_{u}S_{v}=S_{u+v}$
where for any geometric tripotent $w,$ $S_w$ is the symmetry associated with the
symmetric face $F_w$ (see~\cite{fr5}).

The space $\mathbb{R}^{n}$ with the norm
\begin{eqnarray*}
\|f\|=\sum\limits_{i=1}^{k}\sqrt{\sum\limits_{j=n_i}^{n_{i+1}}f_{j}^{2}},
\end{eqnarray*}
where $1=n_1<n_2< \ldots< n_{k+1}=n,$ is a SFS-space of rank $k$
and satisfies JP.

The following theorem gives a description of finite-dimensional
real neutral SFS-spaces with the property JP.

\begin{thm}\label{leibnizsimple}
Let $ Z $ be a finite-dimensional real neutral
strongly facially symmetric space with property JP.
Then
\begin{eqnarray*}
Z\cong \bigoplus\limits_{i=1}^{k} H_i,
\end{eqnarray*}
where $H_i$ is a Hilbert space, $i\in \overline{1,k}$ and $k=\mathrm{rank} Z.$
\end{thm}

Recall that a face $F$ of a convex set $K$ is called a split face if there
exists a face $G,$ called complementary to $F,$ such that $F \cap G = \emptyset$ and $K$
is the direct convex sum $F\oplus_c G,$ i.e. any element $f\in K$ can be uniquely
represented in the form $f= tg +(1-t)h,$ where $t\in[0; 1]$ $f\in F,$ $h\in G.$

In this section we assume that $Z$ is a finite-dimensional real neutral SFS-space with the property JP.

For the proof we need several lemmas.

\begin{lm}\label{lm0}
If $u$ and $v$ are orthogonal geometric tripotents, then

(i) $F_{u+v}=F_u\oplus_c F_{v};$

(ii) $P_2(u+v)=P_2(u)+P_2(v);$

(iii) $P_1(u)P_1(v)=0.$
\end{lm}

\begin{proof} (i) Let $u\diamond v,$  $u, v\in \mathcal{GT}.$
Then, by \cite[lemma 2.5]{fr1} we get that $F_u\diamond F_v.$ Therefore, taking into account that in a real finite-dimensional SFS-space
 $Z$ every norm exposed face
is a simplex (see \cite[proposition 1]{is3}), then
$F_u\oplus_c F_{v}$  is a symmetric
face. Moreover, by \cite[lemma 2.1]{fr1} it follows that $F_u, F_v\subset F_{u+v}$
and therefore $F_u\oplus_c F_{v}\subset F_{u+v}.$
 Then by \cite[lemma 2.7]{fr1}  it follows that $F_u\oplus_c F_{v}=F_{u+v},$
 or $(F_u\oplus_c F_{v})^{\diamond}\cap F_{u+v}\neq \emptyset.$

 Suppose that $0\neq f\in (F_u\oplus_c F_{v})^{\diamond}\cap F_{u+v}.$ Then
 $f\in F_u^{\diamond}$ and $f\in F_{v}^{\diamond}.$ Therefore $f\in F_u^{\diamond}\cap F_{v}^{\diamond}.$
 Then, by \cite[lemma 1.8]{fr1} it follows that $f\in F_u^{\diamond}\cap F_{v}^{\diamond}=F_{u+v}^{\diamond}.$
  It
contradicts the fact that $f\in F_{u+v}.$ Hence, $F_{u+v}=F_u\oplus_c F_{v}.$

(ii)
Since $ Z $ is a finite-dimensional SFS-space, then from (i) follows  (ii).

 (iii) Let $u\diamond v,$  $u, v\in \mathcal{GT}.$ Since $Z$ has the property of JP, then by \cite[Remark 4.2]{fr5} it follows that
\begin{eqnarray*}
P_2(u+v)=P_2(u)+P_2(v)+P_1(u)P_1(v).
\end{eqnarray*}
On the other hand, from (ii)
we have
\begin{eqnarray*}
P_2(u+v)=P_2(u)+P_2(v).
\end{eqnarray*}
This means that $P_1(u)P_1(v)=0.$ The proof is complete.
\end{proof}

\begin{lm}\label{spi}
For every $u\in \mathcal{GT},$ we have
\begin{eqnarray*}
Z_0(u) \diamond (Z_1(u)+Z_2(u)).
\end{eqnarray*}
\end{lm}

\begin{proof}  Let $f\in Z_1(u)$ and $g\in Z_0(u),$ then by
\cite[Corollary 2.2]{fr4} it follows that $v_f\in U_1(u)$ and $v_g\in U_0(u).$
Then by \cite[Corollary 3.4]{fr2} it follows that  $P_2(v_g)^{*}P_1(u)^{*}=0,$
and by lemma 1 we have $P_1(v_g)^{*}P_1(u)^{*}=0.$ Using these relations, we obtain
\begin{eqnarray*}
v_f=P_1(u)^{*}v_f=[P_2(v_g)^{*}+P_1(v_g)^{*}+P_0(v_g)^{*}]P_1(u)^{*}v_f =\\
=P_2(v_g)^{*}P_1(u)^{*}v_f+P_1(v_g)^{*}P_1(u)^{*}v_f+P_0(v_g)^{*}P_1(u)^{*}v_f=\\
=P_0(v_g)^{*}P_1(u)^{*}v_f=P_0(v_g)^{*}v_f.
\end{eqnarray*}
Hence $v_f\in U_0(v_g),$ and therefore, by the orthogonality of geometric tripotents, we have $v_f\diamond v_g.$
 Hence, $f\diamond g.$ This means that  $Z_1(u) \diamond Z_0(u).$
  In addition, by  \cite[proposition 1.5]{fr2} it follows that $Z_2(u)\diamond Z_0(u).$ Therefore, by  \cite[proposition 1.1]{fr2}
$Z_0(u) \diamond (Z_1(u)+Z_2(u)).$
 The proof is complete.
\end{proof}

A geometric tripotent $u$ is called maximal if  $v\in \mathcal{GT}$ from $u\leq v$ implies $v=u$
or is equivalent to $P_0(u)=0$.

\begin{lm}\label{lm1}
Let $u, v\in \mathcal{GT}$ and $u\diamond v.$
If $u+v$ is maximal, then

(i) $P_0(u)P_0(v)=0;$

(ii) $Z_0(u)=Z_2(v)+Z_1(v);$

(iii) $Z=Z_0(u) \oplus Z_0(v).$
\end{lm}

\begin{proof} (i) If $u$ and $v$ are mutually orthogonal
geometric tripotents, it follows from \cite[lemma 1.8]{fr2} that $P_0(u)P_0(v)=P_0(u+v).$
Therefore, from the maximality $u+v$ we have $P_0(u)P_0(v)=P_0(u+v)=0.$

(ii) Let $f\in Z_0(u).$ Then by \cite[Corollary 3.4]{fr2},
\begin{eqnarray*}
f&=&[P_2(v)+P_1(v)+P_0(v)]P_0(u)f=\\
&=&P_2(v)P_0(u)f+P_1(v)P_0(u)f+P_0(v)P_0(u)f=\\
&=&P_2(v)f+P_1(v)f.
\end{eqnarray*}
Thus $f\in (Z_2(v)+Z_1(v)).$

Let $f\in Z_2(v)+Z_1(v).$ Since $P_1(u)P_1(v)=0$ and $P_0(u)P_0(v)=0,$
then by  \cite[Corollary 3.4 (a)]{fr2} we have that
\begin{eqnarray*}
f&=&[P_2(u)+P_0(u)+P_1(u)](P_2(v)f+P_1(v)f)=\\
&=&P_2(u)P_2(v)f+P_0(u)P_2(v)f+P_1(u)P_2(v)f+\\
&+&P_2(u)P_1(v)f+P_0(u)P_1(v)f+P_1(u)P_1(v)f=\\
&=&P_0(u)P_2(v)f+P_0(u)P_1(v)f=P_0(u)[P_2(v)+P_1(v)]f=P_0(u)f.
\end{eqnarray*}
Therefore, $f\in Z_0(u).$ Hence, $Z_0(u)=Z_2(v)+Z_1(v).$

(iii) Since $(Z_2(v)+Z_1(v))\diamond Z_0(v)$ and $Z_0(u)=Z_2(v)+Z_1(v),$ then the Pierce expansion can be rewritten as
$Z=Z_0(u) \oplus Z_0(v).$
The proof is complete.
\end{proof}

\begin{lm}\label{lm1}
Let $Z$ be a neutral strongly facially symmetric space with the property JP.
Then for every $u\in \mathcal{GT},$ the subspace $Z_0(u)$ has the property JP.
\end{lm}

\begin{proof}
Let $u\in \mathcal{GT},$ then by \cite[Proposition 4.1]{fr2} the subspace
$Z_0(u)$ is a neutral SFS-space.
We denote by $\mathcal{GT}_{Z_0(u)}$ and $\mathcal{SF}_{Z_0(u)}$
the set of all geometric tripotents of $U_0(u)$ and symmetric faces in $Z_0(u).$

Let $u_1, u_2\in \mathcal{GT}_{Z_0(u)}$ and $u_1\diamond u_2.$
By the Hahn-Banach theorem, we extend the functionals
$u_1, u_2$ to $Z,$ which we denote by $x_1, x_2$, preserving the norm, respectively.
Then
\begin{eqnarray}\label{autoloc}
\|x_i\|=1, \,\,\, x_i(f)=u_i(f), \, \, \,  \forall f\in Z_0(u),
\,\,\, x_i|_{Z_2(u)+Z_1(u)}=0, \,\,\, i=1,2.
\end{eqnarray}
Thus, by  \cite[Theorem 2.3]{fr2}, there exist geometric
tripotents
 $v_i\in \mathcal{GT}\cap U_0(u)$ such that $F_{x_i}=F_{v_i},$ where $i=1,2.$
 Then by (1), we have
 \begin{eqnarray}\label{autoloc}
 F_{v_i}=F_{u_i}, \,\,\, v_i(f)=u_i(f) \,\,\, \forall f\in Z_0(u), \,\,\,
v_i|_{Z_2(u)+Z_1(u)}=0, \,\,\, i=1,2.
\end{eqnarray}
Further \cite[Theorem 3.6 and Proposition 4.1]{fr2}, imply that
\begin{eqnarray}\label{autoloc}
P_k(u_i):=P_k(v_i)|_{Z_0(u)}, \,\,\,  S_{u_i}:=S_{v_i}|_{Z_0(u)},
\end{eqnarray}
where $k=0,1,2,$ $i=1,2.$

Now we show that $v_1\diamond v_2$ and $S_{u_1+u_2}=S_{v_1+v_2}|_{Z_0(u)}.$
Since $u_1\diamond u_2,$ $v_i\in \mathcal{GT}\cap U_0(u),$ then by virtue of \cite [Theorem 3.6 and Corollary
3.4(a)]{fr2} we have $P_2(u_2)P_2(u_1)=P_2(u_1)P_2(u_2)=0$ and
$P_0(u)P_2(v_i)=P_2(v_i)P_0(u)=P_2(v_i),$ where $i=1,2.$
Therefore, by (2) and (3), for any $f\in Z$ we have
\begin{eqnarray*}
\langle P_2(v_1)^{\ast}v_2, f\rangle &=&\langle v_2, P_2(v_1)f\rangle=\langle u_2, P_2(v_1)f\rangle=\\
&=&\langle P_2(u_2)^{\ast}u_2, P_2(v_1)f\rangle=\langle u_2, P_2(u_2)P_2(v_1)P_0(u)f\rangle=\\
&=&\langle u_2, P_2(u_2)P_2(u_1)P_0(u)f\rangle=\langle u_2, 0\rangle=0.
\end{eqnarray*}
Then $P_2(v_1)^{\ast}v_2=0.$ Analogously, we get that $P_1(v_1)^{\ast}v_2=0.$
Hence, $P_0(v_1)^{\ast}v_2=v_2.$ Therefore, $v_1\diamond v_2.$
Then by (2) the geometric tripotent $v_1+v_2$ is an extension of $u_1+u_2$ and $F_{v_1+v_2}= F_{u_1+u_2}.$
Thus, by virtue of \cite[Theorem 3.6 and Proposition 4.1]{fr2}, we have
\begin{eqnarray}\label{autoloc}
S_{u_1+u_2}:=S_{v_1+v_2}|_{Z_0(u)}.
\end{eqnarray}
Since $Z$ has the property of JP, it follows from (3) and (4) that we have
\[S_{u_1+u_2}=S_{v_1+v_2}|_{Z_0(u)}=S_{v_1}|_{Z_0(u)}S_{v_2}|_{Z_0(u)}=S_{u_1}S_{u_2}.\]
Hence, $Z_0(u)$ has the property JP. The proof is complete.
\end{proof}

\textit{Proof of Theorem~\ref{leibnizsimple}.} In the case when
$dim Z=2$ and $dim Z=3,$  the statement of  theorem follows from \cite[Proposition 1 and Theorem 2]{is3}.

Suppose that the theorem is true when  $\mbox{dim}
Z<n.$  Let $\mbox{dim} Z=n$ and  $\mbox{rank} Z=k.$ Then by the definition of the rank of the space there exists
a maximal family of mutually orthogonal geometric tripotents $u_1, u_2, ..., u_k,$
and by  \cite[lemma 2.5]{fr1}  follows that  $u_1+u_2+...+u_k$ is a maximal geometric tripotent.
Then by lemma 3 it follows that
\begin{eqnarray}\label{autoloc}
Z=Z_0(u) \oplus Z_0(v), \,\,\, Z_0(u) \diamond Z_0(v),
\end{eqnarray}
where $u=u_1, v=u_2+...+u_k.$

By \cite[Proposition 4.1]{fr2} and Lemma 4 the subspaces $Z_0(u)$ and $Z_0(v)$
are neutral SFS-spaces satisfying JP with dimensions less that \(n.\) By the
assumption of induction these are direct sums of Hilbert spaces.
By  (5) it follows that  $Z$ is also a direct sum of Hilbert
spaces. The proof is complete.

Thus, unlike the case when $\mbox{rank} Z = 1,$  $n-1, n$ (in each
of these of cases, up to isometrically isomorphism, there is one
space), in the case $1< \mbox{rank} Z <n-1$ the number of
non-isomorphic spaces is at least  $[\frac{n}{k}],$ where
$k=\mbox{rank} Z,$ $[t]$ is the integer part of the number  $t.$
The number of non-isomorphic spaces can be calculated from the
recurrence relation
\begin{eqnarray*}
p(n)=\sum\limits_{m=1}^{\infty}(-1)^{m+1}\left(p\left(n-\frac{m(3m-1)}{2}\right)+p\left(n-\frac{m(3m+1)}{2}\right)\right).
\end{eqnarray*}

\section{Description of facially symmetric spaces with unitary tripotents}

A geometric
tripotent $e\in \mathcal{GT}$ is said to be unitary if the convex
hull of the set $F_e\cup F_{-e}$ coincides with the unit ball $Z_1,$ i.e.
\begin{eqnarray}\label{autoloc}
Z_1=\mathrm{co}(F_e\cup F_{-e}).
\end{eqnarray}
Also note that property (6) is much stronger than the Jordan decomposition
property of a norm exposed face (see ~\cite[Lemmata 2.3]{fr6}). Recall that a norm exposed face $F_u$
has the Jordan decomposition property if its real span coincides with the
geometric Peirce 2-space of the geometric tripotent $u.$

The space $\mathbb{R}^{n}$ with the norm $\|x\|=\sum\limits_{i=1}^{n}|x_i|,$ $x=(x_i)\in \mathbb{R}^{n}$ is a SFS space.
If $e\in \mathbb{R}^{n}\cong (\mathbb{R}^{n})^{\ast}$ is a maximal geometric tripotent then
 $e=(\varepsilon_1, ..., \varepsilon_n), \varepsilon_i\in \{-1, 1\}, i=\overline{1,n},$ and in this case the norm exposed face
\begin{eqnarray*}
F_e=\left\{x\in \mathbb{R}^{n}: \sum\limits_{i=1}^{n}\varepsilon_ix_i=1, \varepsilon_ix_i\geq 0, i=\overline{1,n} \right\}
\end{eqnarray*}
satisfies (6).

More generally, consider a measure space $(\Omega, \Sigma, \mu)$ with measure $\mu$ having
the direct sum property, i.e. there is a family $\{\Omega\}_{i\in J}\subset \Sigma, 0<\mu(\Omega_i)<\infty, i\in J,$
such that for any $A\in \Sigma$ with $\mu(\Omega_i)<\infty,$ there exist a countable
subset $J_0\subset J$ and a set $B$ of zero measure such that $A=\bigcup_{i\in J_0}(A\cap \Omega_i)\cup B.$

Let $L_1(\Omega, \Sigma, \mu)$ be the space of all real integrable functions on $(\Omega, \Sigma, \mu).$
The space $L_1(\Omega, \Sigma, \mu)$ with the norm $\|f\|=\int\limits_{\Omega}|f(t)|d\mu(t),$
$f\in L_1(\Omega, \Sigma, \mu),$ is a SFS space.
If $e\in L^{\infty}(\Omega, \Sigma, \mu)\cong L_1(\Omega, \Sigma, \mu)^{\ast},$ is a maximal geometric
tripotent then $e=\widetilde{\chi}_A-\widetilde{\chi}_{\Omega\setminus A}$ for some $A\in \Sigma,$ where $\widetilde{\chi}_A$
is the class containing the indicator function of the set $A\in \Sigma.$ Then the norm exposed face
\begin{eqnarray*}
F_e=\left\{f\in L_1(\Omega, \Sigma, \mu): \|f\|=1, \int\limits_{\Omega}e(t)f(t)d\mu(t)=1\right\}
\end{eqnarray*}
satisfies (6).

In \cite{yik}, Theorem 3.2 it was established that if $Z$ is a real neutral strongly
facially symmetric space with unitary tripotent such that any maximal tripotent is unitary, then $Z$ is isometrically isomorphic to $L_1$-space.

The following result shows that the condition "every maximal
triptotent is unitary" in the above theorem is superfluous.

\begin{lm}\label{filider}
Let $Z$ is  a real neutral
SFS-space with unitary tripotent. Then every
maximal tripotent is unitary.
\end{lm}

\begin{proof}   Let $e$ be a unitary geometric tripotent and $u$ is maximal geometric tripotent. Then
by \cite[Lemma 3.4]{yik}, there exist mutually orthogonal geometric tripotents $u_1, u_2\leq e$ such that  $u=u_1-u_2.$
Since $u$ is  a maximal tripotent, then by \cite[Lemma 3.5]{yik} we get  that $u_1+u_2$ is a maximal geometric tripotent.

On the other side, if  $u_1, u_2\leq e$
are mutually orthogonal geometric tripotents, then by \cite [Lemma 2.1]{fr1} we obtain that
$F_{u_1}, F_{u_2} \subset F_{u_1+u_2},$ consequently $u_1, u_2 \leq u_1+u_2.$
Since, $L_e=\{v\in \mathcal{GU}: v\leq e\}\cup \{0\}$
is a complete orthomodular lattice
(see ~\cite[Proposition 4.5]{fr2}), then $u_1+u_2\leq e.$
 Therefore, by the maximality of $u_1+u_2$ it follows  that $e=u_1+u_2.$ Then $F_{u_1}=F_u\cap F_e$ and $F_{-u_2}=F_u\cap F_{-e}.$

Now, we show that
\begin{eqnarray}\label{autoloc}
F_u=F_{u_1}\oplus_c F_{-u_2}.
\end{eqnarray}
For this it suffices to show that $F_u\subseteq F_{u_1}\oplus_c F_{-u_2}.$

Let $f\in F_u.$ Then, from equality $(6)$ we obtain $f=tg+(1-t)h$ where $g\in F_e,$ $h\in F_{-e},$ $0\leq t\leq 1.$

If $t=1,$ then $f=g,$ and therefore $f\in F_u\cap F_e=F_{u_1}.$

If $t=0,$ then $f=h,$ and therefore $f\in F_u\cap F_{-e}=F_{-u_2}.$

Now let $0<t<1.$ Since $F_u$ is a face, then $g, h\in F_u.$ Therefore $g\in F_u\cap F_e=F_{u_1},$
$h\in F_u\cap F_{-e}=F_{-u_2}.$ Thus, from $F_{u_1}\diamond F_{-u_2}$  we obtain $f\in F_{u_1}\oplus_c F_{-u_2}.$
Therefore,
\begin{eqnarray*}
F_u= F_{u_1}\oplus_c F_{-u_2}.
\end{eqnarray*}

Let $f\in F_{u_1+u_2}.$ Then $\|P_2(u_1-u_2)f\|\leq\|f\|$ and, in view of \cite [Lemma 4.2]{fr2} we obtain
\begin{eqnarray*}
\langle P_2(u_1-u_2)f, u_1+u_2\rangle=\langle f, P_2(u_1-u_2)^{\ast}(u_1+u_2)\rangle=\\
=\langle f, P_2(u_1-u_2)^{\ast}u_1\rangle+\langle f, P_2(u_1-u_2)^{\ast}u_2\rangle=\langle f, u_1+u_2\rangle=1
\end{eqnarray*}
Therefore, $\|P_2(u_1-u_2)f\|=\|f\|$ and from the neutrality of $P_2(u_1-u_2)f=f.$
Thus, from the equalities (6) and (7) we obtain
\begin{eqnarray*}
Z=\mathrm{\overline{sp}}F_e=\mathrm{\overline{sp}}F_u=\mathrm{\overline{sp}}(F_{u_1}\oplus_c F_{-u_2})
=\mathrm{\overline{sp}}F_{u_1}\oplus \mathrm{\overline{sp}}F_{-u_2}=\mathrm{\overline{sp}}F_{u_1}\oplus \mathrm{\overline{sp}}F_{u_2}
\end{eqnarray*}
Therefore, using equalities  (6) and (7) we obtain that
\begin{eqnarray*}
\mathrm{co}(F_u\cup F_{-u})=\mathrm{co}((F_{u_1}\oplus_c F_{-u_2})\cup (F_{-u_1}\oplus_c F_{u_2}))=\\
=\mathrm{co}((F_{u_1}\oplus_c F_{u_2})\cup (F_{-u_1}\oplus_c F_{-u_2}))=\mathrm{co}(F_e\cup F_{-e})=Z_1.
\end{eqnarray*}
This means that $u$ is unitary. The proof is
complete.
\end{proof}

Now we can formulate \cite [Theorem 3.1]{yik} as follows.

\begin{thm}\label{pure}
Let $Z$ be a real neutral strongly facially symmetric space with
unitary geometric tripotents. Then there exists a measure space
$(\Omega, \Sigma, \mu)$ with measure $\mu$ having the direct sum
property such that the space $Z$ is isometrically isomorphic to
the space  $L_1(\Omega, \Sigma, \mu).$
\end{thm}


\begin{thebibliography}{99.}%
%
%

\bibitem{fr1} Friedman, Y., Russo, B.: A geometric speñtral theorem. Quart. J.
Math. Oxford. (1986) doi: 10.1093/QMATH/37.3.263
\bibitem{fr2} Friedman, Y., Russo, B.: Affine structure of facially symmetric
spaces. Math. Proc. Camb. Philos. Soc. (1989) doi: 10.1017/S030500410006802X
\bibitem{fr3} Friedman, Y., Russo, B.: Some affine geometric aspects of operator
algebras. Pac. J. Math. (1989) doi: 10.2140/pjm.1989.137.123
\bibitem{fr4} Friedman, Y., Russo, B.: Geometry of the dual ball of the spin
factor. Proc. Lon. Math. Soc. (1992) doi: 10.1112/plms/s3-65.1.142
\bibitem{fr5} Friedman, Y., Russo, B.: Classification of atomic facially
symmetric spaces. Canad. J. Math. (1993) doi: 10.4153/CJM-1993-004-0
\bibitem{ikts} Ibragimov, M., Kudaybergenov, K., Tleumuratov, S., Seypullaev, J.: Geometric Description of the Preduals
of Atomic Commutative von Neumann Algebras. Mathematical Notes. (2013) doi: 10.1134/S0001434613050076
\bibitem{is1} Ibragimov, M., Seypullaev, J.:
Geometric properties of the unit ball of an SFS-space of finite rank. Uzb. Math. journal. \textbf{2}, 10--19 (2005)
\bibitem{is2} Ibragimov, M., Seypullaev, J.:
Description of the unit balls of the facially symmetric spaces of small dimension. Bulletin of KarSU. \textbf{2}, 3--5 (2009)
\bibitem{is3} Ibragimov, M., Seypullaev, J.:
Description of n-dimensional real strongly facially symmetric spaces of rank n-1. Uzb. Math. journal. \textbf{4}, 39--46 (2015)
\bibitem{fr6} Neal, M., Russo, B.: State space of JB*-triples. Math. Ann. (2004) doi: 10.1007/s00208-003-0495-9
\bibitem{s} Seypullaev, J.:
Geometric characterization of Hilbert spaces. Uzb. Math. journal. \textbf{2}, 107--112 (2008)
\bibitem{yik} Yadgorov, N., Ibragimov, M., Kudaybergenov, K.: Geometric characterization of $L_1$-spaces. Studia Math. (2013) doi: 10.4064/sm219-2-1
%

\end{thebibliography}
\end{document}